\documentclass{amsart}
\usepackage{units}
\usepackage{amsmath,amssymb,amsfonts}
\usepackage{graphicx}
		  \newcommand{\Rn}{\mathbb{R}^{N}}
		  \newcommand{\Cn}{\mathbb{C}^N}
		  
		  \newcommand{\R}{\mathbb{R}}
		  \newcommand{\C}{\mathbb{C}}
		  \newcommand{\Z}{\mathbb{Z}}
		  \newcommand{\Sn}{\mathrm{S}^{N-1}}

		  \newcommand{\dsum}{\displaystyle\sum}
		  \newcommand{\dprod}{\displaystyle\prod}
		  \newcommand{\dint}{\displaystyle\int}
		  \newcommand{\bx}{\mathbf{x}}

		  \newcommand{\bz}{\mathbf{z}}
		  \newcommand{\bm}{\mathbf{m}}
		  \newcommand{\bn}{\mathbf{n}}	
		    
		  \newcommand{\by}{\mathbf{y}}
		  \newcommand{\lb}{\left\lbrace}
		  \newcommand{\rb}{\right\rbrace}
		  \newcommand{\al}{\left\langle}
		  \newcommand{\ar}{\right\rangle}
		  \newcommand{\ra}{\rightarrow}
		  \newcommand{\bxi}{\boldsymbol{\xi}}
	           \newcommand{\bo}{\boldsymbol{\omega}}
	           
	 	  \newcommand{\hf}{\nicefrac{1}{2}}
	 
		  \newcommand{\ol}{\overline}
		  \newcommand{\ds}{\displaystyle}

		  \newcommand{\btheta}{\boldsymbol{\theta}}
		  
		  \newcommand{\lp}{\left(}
		  \newcommand{\rp}{\right)}

		  \newcommand{\Imaginary}{\mathrm{Im}}

		  \newcommand{\vol}{\mathrm{vol}_{N}}
		  \numberwithin{equation}{section}
		  \newcommand{\mU}{\mathcal{U}}

\newtheorem*{theoremwn}{Theorem}
\newtheorem{theorem}{Theorem}
\newtheorem*{conjecture}{Conjecture}

\newtheorem{proposition}{Proposition}

\theoremstyle{definition}
\newtheorem{definition}{Definition}
\newtheorem{remark}{Remark}
\begin{document}

\title{The Blaschke-Santal\'o Inequality }
\markright{The Blaschke-Santal\'o Inequality}

\author[Bianchi]{Gabriele Bianchi}
\address{Dipartimento di Matematica e Informatica ``U. Dini'', Universit\`a di Firenze  }
\email{gabriele.bianchi@unifi.it}

\author[Kelly]{Michael Kelly}
\address{Department of Mathematics, University of Texas}
\email{mkelly@math.utexas.edu}


\begin{abstract}
The Blaschke-Santal\'o Inequality  is the assertion that the volume product of a centrally symmetric convex body in Euclidean space is maximized by (and only by) ellipsoids. In this paper we give a Fourier analytic proof of this fact.
\end{abstract}

\maketitle

\section{Introduction}
Let $K$ be a  \emph{convex body} in $\Rn$, that is a  compact convex subset of $\Rn$ with non-empty interior, and assume that the origin is an interior point of $K$. We associate to $K$ another convex body $K^{*}$, called the  {\it dual body} or {\it polar body} of $K$, defined by 
	\begin{equation*}\label{dualBody}
		K^{*}=\{ \by\in\Rn : \bx\cdot\by\leq 1 \;\text{ for each }\bx\in K \},
	\end{equation*} 
where $\bx\cdot\by$ is the usual scalar product. The terminology {\it dual body} is fitting, because the unit ball of any norm in $\Rn$ is a convex body and its dual body is the unit ball of the corresponding dual norm. 

Assume that $K$ is origin symmetric, i.e. $K=-K$. The product
		\[
			P(K)=\vol(K)\vol(K^*),
		\]
where $\vol$ denotes $N$-dimensional Lebesgue measure in $\Rn$, is called the \emph{volume product} of $K$. For a general convex body $K$  the volume product $P(K)$ is defined as the minimum, for $x$ in the interior of $K$, of $\vol(K)\vol\left( (K-x)^* \right)$. Here  $K-x$ is the translate of $K$ by $-x$. The functional $P(K)$ is an affine invariant and thus  all ellipsoids in $\Rn$ have the same volume product, and all parallelotopes in $\Rn$ have the same volume product. Furthermore, $P(K^{*})=P(K)$, because $\left(K^*\right)^*=K$, and as a consequence, for instance, the volume product of the unit cube in $\R^3$ (the $\ell_{\infty}$ unit ball) is the same as the volume product of the octahedron (the $\ell_{1}$ unit ball). 
All of these observations were made by Kurt Mahler in the 1930's, in connection to transference principles for linear forms (see Cassels~\cite{C}).

A sharp upper bound for the volume product is given by the Blaschke-Santal\'o Inequality. 
\begin{theoremwn}[The Blaschke-Santal\'o Inequality]\label{santalo}
Let $B$ denote the Euclidean unit ball of $\Rn$. For every convex body $K$ in $\Rn$
\begin{equation}\label{blaschkesantalo}
 P(K)\leq P(B),
\end{equation}
and equality holds if and only if $K$ is an ellipsoid.
\end{theoremwn}
Inequality~\eqref{blaschkesantalo}  was first proved by Blaschke~\cite{Blaschke1917} for $N=2,3$, by Santal\'o~\cite{Santalo1949} for any $N$. Petty~\cite{Petty1985}, completed the proof of the equality case. These results were obtained in the context of affine differential geometry, as  a consequence of results on the affine isoperimetric inequality and of its equality cases.
Later proofs, which are more direct and use classical tools of convexity, are due, among others, to Saint Raymond~\cite{SR}, for origin symmetric bodies, and to Meyer and Pajor~\cite{MeyPaj1990}. See Schneider~\cite{S2} for a  detailed account of the literature on the volume product.

It is our goal in this paper to prove the previous theorem, in the class of origin symmetric convex bodies, using a Fourier analytic approach. See Theorem~\ref{oursantalo} in Section~\ref{proof_blascke_santalo} for the statement. 

The minimum of the volume product for $N\geq3$ is still unknown. It is conjectured that, for a convex body $K$, we have 
\begin{equation}\label{minvolproduct}
 P(K)\geq \frac{(N+1)^{N+1}}{(N!)^2},
\end{equation}
with equality precisely for simplices, and that, for an origin symmetric convex body $K$, we have 
\begin{equation}\label{mahler}
 P(K)\geq \frac{4^N}{N!},
\end{equation}
with equality holding for affine transforms of cubes, of crosspolytopes and, more generally, for Hanner polytopes. 
%The latter inequality~\eqref{mahler} 
Mahler \cite{M} was able to prove~\eqref{minvolproduct} and~\eqref{mahler} when $N=2$. 

Inequality~\eqref{mahler} is known as \emph{Mahler conjecture} and it has remained open for over three quarters of a century.
It has been proved in certain classes of bodies, for instance when $K$ is a zonotope (see Reisner~\cite{Reisner1,Reisner2}) or when $K$ is $1$-unconditional, i.e. an affine transform of $K$ is symmetric with respect to each coordinate hyperplane (see Saint Raymond~\cite{SR}).
Recently it has been proved that the cube (see Petrov et al.~\cite{NPRZ}) and every Hanner polytopes (see Kim~\cite{Kim2013})  are  local minimizers of the volume product (and strict local minimizers in the proper sense) in the class of origin-symmetric convex bodies.
The interested reader is advised to consult   Tao~\cite{Tao} for a very nice discussion about the conjecture and some of its subtleties. 

As far as we know,  the Fourier analytic approach to study the volume product has been first used by F.~Nazarov. He~\cite{N} used it, together with H\"ormander's solution to the $\bar{\partial}$ problem, to prove the Bourgain-Milman Inequality~\cite{BM}
\[
 P(K)\geq c^N\frac{4^N}{N!},
\]
where $c>0$ is a constant not depending on $N$ and $K\subset\Rn$ is an origin symmetric convex body. Ryabogin and Zvavitch~\cite{RyaZva} describes the ideas behind Nazarov's proof as well as those behind the proofs of some of the results on the Mahler conjecture mentioned above.

The main object in our investigation is the following functional.
\begin{definition}\label{dirac}
Given an origin symmetric convex body $K\subset\Rn$ define
	\begin{equation*}%\label{eta}
		\rho(K)=\inf\dint_{\Rn} |F{\bf(x)}|^{2}d\bx,
	\end{equation*}
where the infimum is taken over the class of square-integrable continuous functions $F:\R^N \ra \C$ that satisfy
		\begin{enumerate}
			\item\label{rho_1} $ |F({\bf 0})|\geq 1$, and
			\item\label{rho_2} $\widehat{F}(\bxi)=0$ if $\bxi\in\Rn\setminus K$.
		\end{enumerate}
\end{definition}
As we will prove in Section~\ref{proof_blascke_santalo}, $\rho(K)=1/\vol(K)$ and the only minimizers of $\rho$ are admissible multiple of the inverse Fourier transform of the characteristic function $1_K$ of $K$. 
On the other hand, the Paley-Wiener Theorem (see next section) states that the analytic extension to $\Cn$ of every function $F$ admissible for $\rho$ has an  asymptotic behavior at infinity which is  related to the norm whose unit ball is $K^*$. 
This connection is at the hearth of this proof of the Blaschke-Santal\'o Inequality. 

To deal with the equality cases, we will show that if $K$ is origin symmetric and $P(K)=P(B)$, then, for each direction $\btheta\in\Sn$, there exists an ellipsoid $E$ (which a priori may depend on $\btheta$) such that for each hyperplane $L$ orthogonal to $\btheta$ the $(N-1)$-volume of the sections $K\cap L$ and $E\cap L$ coincide. 
This property, and a result proved by M.~Meyer and S.~Reisner~\cite[Lemma 3]{MR}, imply that $K$ is an ellipsoid.

In section 4 we introduce a variational quantity $\eta(K)$ associated with an origin symmetric convex body. It is essentially an $L^1$ version of $\rho(K)$. We state a conjecture (due to the second author and Jeffrey Vaaler) regarding the exact value of $\eta(K)$ and prove it when $K$ is a ball or a cube.

For another problem in convex geometry where the Fourier transform of $1_K$ in $\C^n$ plays an important role see Bianchi~\cite{Bia13}.

Gabriele Bianchi wishes to acknowledge that all proofs in this paper are due to Michael Kelly, except for that of the equality case in the Blaschke-Santal\'o Inequality, which is due to himself.

\section{Background and Notation} Throughout this paper $z$ denotes an element of the complex numbers $\C$, and $\ol{z}$ denotes the complex conjugate of $z$. The symbol $\mU=\lb z\in\C:\Imaginary(z)>0 \rb$  denotes the upper half plane of $\C$, where $\Imaginary(z)$ is the \emph{imaginary part} of $z$. We use boldface letters or symbols to denote vectors, $\bx$ denotes a vector in $\Rn$, $\bz$ a vector in $\Cn$, and $\Imaginary(\bz)$ denotes the vector of the imaginary parts of $\bz$.  
We write $\mathrm{vol}_k$ for $k$-dimensional Lebesgue measure in $\Rn$. By $B$ and $\Sn$ we denote respectively the Euclidean unit ball and unit sphere in $\Rn$. The symbol $\omega_{N-1}$ indicates the surface area of $\Sn$.

The \emph{support function} of a convex body $K$ in $\Rn$ is defined, for $\bx\in\Rn$, by 
\[
h_K(\bx)=\sup \{\bx\cdot\by : \by\in K\}.
\]
If $K$ is an origin symmetric convex body and $\|\cdot\|_{K^*}$ denotes the norm in $\Rn$ whose unit ball is $K^*$, i.e.
$
\|\bx\|_{K^*}=\inf\{\lambda>0 : x\in \lambda K^*\},
$
we have
\begin{equation}\label{support_and_norm}
h_K(\bx)=\|\bx\|_{K^*}.
\end{equation}

Given a convex body $K$ in $\R^n$, $t\in\R$ and $\btheta\in \Sn$,  we denote by $S_K(t,\btheta)$ the \emph{Radon transform} of the characteristic function $1_K$ of $K$
\begin{equation*}%\label{definition_RT}
S_K(t,\btheta)=\mathrm{vol}_{N-1}\left(\left\{\bx\in K : \bx\cdot\btheta=t\right\}\right).
\end{equation*}

For a function $F\in L^2(\Rn)$ the \emph{Fourier transform} $ \widehat{F}$ is defined for $\bxi\in\Rn$ by
	\begin{equation*}
		\widehat{F}(\bxi)=\ds\lim_{T\ra\infty}\dint_{[-T,T]^{N}} e^{-2\pi i {\bx\cdot\bxi }} F{\bf (x)}d\bx.
	\end{equation*}

A function $ F:\Cn\ra \C$ is an \emph{entire function} if it is holomorphic, in each coordinate separately,  at each $\bz\in\Cn$. If $F$ is an entire function, the complex conjugate $ F^{*}$ of $F$, defined by $F^{*}(\bz)=\ol{F(\ol{\bz})}$, is also an entire function.

Let $K$ be an origin symmetric convex body in $\Rn$. Following Stein and Weiss \cite[\S 3.4]{SW} we call an entire function $F$ \emph{of exponential type $K^*$} if for every $\epsilon>0$ there exists a constant $c_{\epsilon}>0$ such that, for every $\by\in\Rn$,
	\begin{equation}\label{growthEstimate}
		|F(i \by)|\leq c_{\epsilon} e^{2\pi (1+\epsilon) \|\by\|_{K^*}}.
	\end{equation}
When $F\in L^2(\Rn)$ is such that the support of $\widehat{F}$ is contained in $K$ then it is well known that $F$ is the restriction to $\Rn$ of the entire function defined, for $\bz\in\Cn$, by the formula 
\begin{equation}\label{inversion_formula}
 F(\bz)=\int_K e^{2\pi i \bz\cdot\bxi} \widehat{F}(\bxi) d\bxi.
\end{equation}
This representation and the Cauchy-Schwarz Inequality imply 
\begin{equation*}%\label{eq:mono}
%\begin{aligned}
|F(i\by)|\leq\int_K \left| e^{-2\pi \by\cdot\bxi}\widehat{F}(\bxi)\right| d\bxi
%&\leq e^{2\pi h_K(y)}\int_K e^{2\pi i \bx\cdot\bxi} \widehat{F}(\bxi) d\bxi\\
\leq \vol(K)^{1/2}\left(\int_K |\widehat{F}(\bxi)|^2 d\bxi\right)^{1/2}e^{2\pi h_K(\by)},
%\end{aligned}
\end{equation*}
i.e., in view of~\eqref{support_and_norm}, they imply that $F$ is of exponential type $K^*$. The following theorem (due to Paley and Wiener in the one dimensional case and Stein in the general case) proves that these properties are equivalent.
\begin{theoremwn}[Paley-Wiener-Stein~\cite{SW}]
Let $F\in L^2(\Rn)$ and let $K$ be an origin symmetric convex body. Then $F$ is a.e. equal to the restriction to $\R^N$ of an entire function of exponential type $K^*$ if and only if  the support of $\widehat{F}$ is contained in $K$.
\end{theoremwn}

\section{Proof of the Blaschke-Santal\'o Inequality}\label{proof_blascke_santalo}
\begin{theorem}\label{oursantalo}
For every origin symmetric convex body $K$ in $\Rn$
\begin{equation}\label{ourblaschkesantalo}
 P(K)\leq P(B),
\end{equation}
and equality holds if and only if $K$ is an ellipsoid.
\end{theorem}

\begin{proof}Our proof of \eqref{ourblaschkesantalo} proceeds in two parts. First we show that
	      \begin{equation}\label{rhoeval}
		    \rho(K)=\dfrac{1}{\vol(K)}
	      \end{equation}
and then we  show that
	      \begin{equation}\label{rhoineq}
		    \dfrac{\rho(B)}{\vol(B^{*})}\leq \dfrac{\rho(K)}{\vol(K^{*})}.
	      \end{equation}
Plugging (\ref{rhoeval}) into (\ref{rhoineq}) and rearranging terms yields~\eqref{ourblaschkesantalo}.

We say that a continuous $F\in L^2(\Rn)$ is \emph{admissible for $\rho(K)$} provided that $F$ satisfies conditions \eqref{rho_1} and \eqref{rho_2} in the definition of $\rho$.

{\bf Let us  prove~\eqref{rhoeval}.} Let $F(\bx)$ be an admissible function for $\rho(K)$. Condition \eqref{rho_1} is equivalent to $\left|\int_K \widehat{F}(\bxi)d\bxi\right|\geq1$, due to formula \eqref{inversion_formula}. We can thus write
\begin{equation}\label{bs:0}
\begin{aligned}
1&\leq\left|\int_K \widehat{F}(\bxi)d\bxi\right|^2\\
&\leq \vol(K)\int_K \left|\widehat{F}(\bxi)\right|^2 d\bxi\\
&=\vol(K) \int_{\Rn} \left|{F}(\bx)\right|^2 d\bx\\
&\leq  \vol(K)\rho(K).
\end{aligned}
\end{equation}
The inequality in the second line is a consequence of Cauchy-Schwarz Inequality, and the equality in the third line is a consequence of Parseval's Identity. 
Note that, by the discussion of the equality cases in Cauchy-Schwarz Inequality, $F$ minimizes $\rho(K)$ if and only if $F$ is an admissible multiple of the inverse Fourier transform of $1_K$, i.e.
\begin{equation}\label{F_equality}
 F(\bx)=\frac{\alpha}{\vol(K)}\int_K e^{2\pi i \bx\cdot\bxi}d\bxi
\end{equation}
for some $\alpha\in\C$, with $|\alpha|=1$. 
This concludes the proof of~\eqref{rhoeval}.

\indent {\bf Now let us prove~\eqref{rhoineq}.} Let $F(\bx)$ be an admissible function for $\rho(K)$. Without loss of generality we may assume that $F(\bx)$ is even. This is because the even part of $F(\bx)$ is admissible for $\rho(K)$ and (by the triangle inequality) has a $L^2$-norm less than or equal to that of $F(\bx)$. Let us denote by $F$ also the entire extension of $F$ defined by~\eqref{inversion_formula}.

For each $\btheta\in\Sn,$ we define a function $G_{\btheta}:\C\to\C$ as
		\[
			 	G_{\btheta}(z)= F(z\btheta).
		\]
This is an even entire function of exponential type $[-\|\btheta\|_{K^*}^{-1},\|\btheta\|_{K^*}^{-1}]$, by (\ref{growthEstimate}). Note that, since $G_{\btheta}(z)$ is even, there exists an entire function $H_{\btheta}(z)$ such that $G_{\btheta}(z)=H_{\btheta}(z^{2})$. Finally we define $ R_{\btheta}:\Cn\to\C$ as the radial extension of $G_{\btheta}(z)$, i.e., as
	\[
		R_{\btheta}(\bz)=H_{\btheta}\left( z_{1}^{2}+\cdots +z_{N}^{2} \right).
	\] 
By Fubini's Theorem, $\int_0^{+\infty}|F(r\btheta)|^2r^{N-1} dr$ exists finite almost for every $\btheta\in\Sn$, i.e. the restriction of $R_{\btheta}$ to $\Rn$ is square-summable almost for every $\btheta\in\Sn$. 
We clearly have 
\begin{equation}\label{bs:1}
\begin{aligned}
\dint_{\Rn}  |F(\bx)|^{2}d\bx &=  \dint_{\Sn}\dint_{0}^{\infty}  |F(r\btheta)|^{2} r^{N-1}dr\ d\sigma(\btheta)\\
&= \dfrac{1}{\omega_{N-1}}\dint_{\Sn}\dint_{\Rn} |R_{\btheta}(\bx)|^{2} d\bx\ d\sigma(\btheta),
\end{aligned}
\end{equation}
where  $d\sigma$ is the standard surface measure on  $\Sn$.

The function $R_{\btheta}(\bz)$ satisfies $|R_{\btheta}(\bf 0)|\geq 1$, it is entire and the support of the Fourier transform of the restriction of $R_{\btheta}$ to $\Rn$ is contained in the ball $\|\btheta\|_{K^*} B$. The last claim is a consequence of the Paley-Wiener Theorem, of the fact that $R_{\btheta}(\bz)$ is of exponential type $\|\btheta\|_{K^*}^{-1} B$ and of  $\|\btheta\|_{K^*}^{-1} B=(\|\btheta\|_{K^*} B)^*$. 
The function $R_{\btheta }(\bx)$ is thus admissible for $ \rho\left(\|\btheta\|_{K^*} B\right)$. Therefore 
\begin{equation}\label{bs:2}
\begin{aligned}
\dint_{\Rn} |R_{\btheta}(\bx)|^{2} d\bx&\geq \rho\left(\|\btheta\|_{K^*} B\right)\\
&=\|\btheta\|_{K^*}^{-N} \rho(B),
\end{aligned}
\end{equation}
since $\rho$ is positively homogeneous of degree $-N$.

The set $K^*$ can be represented in polar coordinates as 
\[
 K^*=\{\rho\btheta : \btheta\in\Sn, 0\leq\rho\leq\|\btheta\|_{K^*}^{-1}\},
\]
and therefore
\begin{equation}\label{dualVolumeInt}
		\dfrac{\vol(K^{*})}{\vol(B^{*})}= \dfrac{1}{\omega_{N-1}}\dint_{\Sn}\|\btheta\|_{K^*}^{-N} d\sigma(\btheta).
\end{equation}

Using \eqref{bs:1}, \eqref{bs:2} and~\eqref{dualVolumeInt} we obtain
\begin{equation}\label{bs:3}
\dint_{\Rn}  |F(\bx)|^{2}d\bx\geq \rho(B)\dfrac{\vol(K^{*})}{\vol(B^{*})}.
\end{equation}
The inequality~\eqref{rhoineq} then follows upon taking the infimum over all admissible functions $F(\bx)$.

\indent {\bf Let us now prove that we have equality in~\eqref{ourblaschkesantalo} only when $K$ is an ellipsoid.} 
Let $F(x)$ be as in~\eqref{F_equality}, with $\alpha=1$. This function is admissible for $\rho(K)$, is even and $\int_{\Rn} |F(x)|^2 d\bx=\rho(K)$. Therefore equality holds in~\eqref{ourblaschkesantalo} if and only if equality holds in~\eqref{bs:3}.
The proof of \eqref{bs:3} reveals that this happens if and only if $R_{\btheta}(\bx)$ minimizes $\rho(\|\btheta\|_{K^*} B)$ almost for every $\btheta\in\Sn$. 
In view of the discussion at the end of the proof of~\eqref{rhoeval} and of the definition of $R_{\btheta}$, this is equivalent to saying that  almost for every $\btheta\in\Sn$ and for every $r\geq0$, $F(r\btheta)$ coincides with an admissible multiple of the restriction to the ray $\{r\theta : r\geq0\}$ of the inverse Fourier transform of $1_{\|\btheta\|_{K^*}B}$. 
This is equivalent to saying that there exists $\alpha(\theta)\in\C$ with $|\alpha(\theta)|=1$ such that for each $r\in\R$ 
  \begin{equation}\label{bs:4}
   \frac{1}{\vol(K)}\int_K e^{2\pi i r\btheta\cdot\bxi}d\bxi=\frac{\alpha(\btheta)}{\vol(\|\btheta\|_{K^*}B)}\int_{\|\btheta\|_{K^*}B} e^{2\pi i r\btheta\cdot\bxi}d\bxi.
  \end{equation}
Since, by Fubini's Theorem, the $n$-dimensional inverse Fourier transform of $1_K$ is the $1$-dimensional inverse Fourier transform of the Radon transform $S_K$ of $K$, \eqref{bs:4} can be rewritten   as
\begin{equation*}%\label{F_equality}
   \frac{1}{\vol(K)}
   \int_\R e^{2\pi i rt}S_K(t,\btheta)\ dt
   =\frac{\alpha(\btheta)}{\vol(\|\btheta\|_{K^*}B)}
   \int_\R e^{2\pi i rt}S_{\|\btheta\|_{K^*} B}(t,\btheta)\ dt.
  \end{equation*}
This identity implies that for each $t\in\R$ and almost for every $\btheta\in\Sn$
\begin{equation}\label{bs:5}
 \frac{1}{\vol(K)} S_K(t,\btheta)
 =\frac{\alpha(\btheta)}{\vol(\|\btheta\|_{K^*}B)} S_{\|\btheta\|_{K^*} B}(t,\btheta).
\end{equation}
By continuity the previous identity holds for each $\btheta\in\Sn$. Moreover, since $|\alpha(\btheta)|=1$ and each other term in~\eqref{bs:5} is non-negative, we have $\alpha(\btheta)=1$. 

For  $\btheta\in S^{N-1}$ and $t\in\R$ let
\[
D_K(t,\btheta)=\vol\left(
\left\{\bx\in K : \bx\cdot\btheta\geq t\  \|\btheta\|_{K^*} \right\}\right).
\]
Meyer and Reisner~\cite[Lemma 3]{MR} proves that  if  $D_K(t,\btheta)$ does not depend on $\btheta$ for each $t\in[0,1]$ then $K$ is  an ellipsoid. We prove that this is the case. 
We write
\begin{equation} \label{DS}
	D_K(t,\btheta)=\dint_{t \|\btheta\|_{K^*}}^{\|\btheta\|_{K^*}}S_{K}(r,\btheta) dr.
\end{equation}
Formula~\eqref{bs:5} implies	
\begin{equation*}
\begin{aligned}
D_K(t,\btheta)
&=\frac{\vol(K)}{h_K(\theta)^N\vol(B)} 
\dint_{t \|\btheta\|_{K^*}}^{\|\btheta\|_{K^*}}S_{\|\btheta\|_{K^*} B}(r,\btheta) dr \\
&=\frac{\omega_{N-2}\vol(K)}{h_K(\theta)^N\vol(B)} 
\dint_{t \|\btheta\|_{K^*}}^{\|\btheta\|_{K^*}}  \left(\|\btheta\|_{K^*}^2-r^2\right)^{\frac{N-1}2} dr\\
&=\frac{\omega_{N-2}\vol(K)}{\vol(B)} 
\dint_t^1  \left(1-s^2\right)^{\frac{N-1}2} ds.
\end{aligned}	
\end{equation*} 
This concludes the proof.
\end{proof}
\begin{remark}
The validity of~\eqref{bs:5}, with $\alpha(\btheta)=1$, for a given $\btheta$ and for each $t\in\R$ is equivalent to the existence of an ellipsoid $E(\btheta)$ such that $S_K(t,\btheta)=S_{E(\btheta)}(t,\btheta)$ for each $t\in\R$.
\end{remark}

\section{A related variational quantity}
Another extremal quantity related to $\rho(K)$ is the ``$L^{1}-$version'' $\eta(K)$. 
\begin{definition}
 Given a convex body $K$ define
	\begin{equation*}%\label{eta}
		\eta(K)=\inf\dint_{\Rn} F{\bf(x)}d\bx
	\end{equation*}
where the infimum is taken over the class of non-zero continuous functions $F(\bx)$ that satisfy
		\begin{enumerate}
			\item $F(\bx)\geq 0$ for every $\bx\in\Rn$,
			\item $ F({\bf 0})\geq 1$, and
			\item $\widehat{F}(\bxi)=0$ if $\bxi\in\Rn\setminus  K$.
		\end{enumerate}
\end{definition}
When $K$ is a cube, the infimum is achieved by the Fej\'er kernel. An extremal function for a generic origin symmetric convex body $K$ can then be thought of as a ``Fej\'er kernel associated with $K$.'' On another level, the determination of $\eta(K)$ is perhaps the simplest form of the so-called {\it Beurling-Selberg extremal problem} in several variables. The difficulty in determining $\eta(K)$ is the non-negativity, which is awkward from the Fourier analytic point of view. In the single variable case the function $F(x)$ can be factored as $F(x)=|U(x)|^{2}$ where $U(x)$ is admissible for $\rho(K/2)$. In several variables such a factorization is not generally available, and is known to be false for trigonometric polynomials of two or more variables. 
%\begin{figure}[h]
%\begin{center}
%	\includegraphics[scale=.4]{fejer.png}\hspace*{1.4cm}
%\end{center}
%	\caption{The extremal function for $\eta(K)$ in 1 dimension is the Fejer kernel.}
%\end{figure}
	However, Jeff Vaaler and the second author conjecture that there are extremal functions for $\eta(K)$that do admit such a factorization. 
\begin{conjecture}%\mcomment{are these conjectures equivalent? If this is the case it is not clear}
	For any origin symmetric convex body $K\subset \Rn$, we have 
		\[
			\eta(K)=\dfrac{2^{N}}{\vol(K)}.
		\]
\end{conjecture}
\begin{remark}
	From~\eqref{rhoeval} it follows that $\eta(K)\leq \rho(K/2)=2^{N}\rho(K)=2^{N}\vol(K)^{-1}$. The above conjecture asserts that there is equality in this inequality for {\it every} origin symmetric convex body $K$.% \mcomment{It is not clear to me why $\eta(K)\leq \rho(K/2)$}
\end{remark}

Our main goal in this section is to prove that this conjecture holds when $K$ is a ball and when $K$ is a cube. 
\begin{theorem} Let $B\subset \Rn$ be the Euclidean unit ball and $Q\subset \Rn$ be the Euclidean unit cube. Then
		\begin{equation}\label{etaball}
		\eta(B)=\dfrac{2^{N}}{\vol(B)}
	\end{equation}
	and
		\begin{equation}\label{etaq}
		\eta(Q)=\dfrac{2^{N}}{\vol(Q)}.
	\end{equation}
\end{theorem}
The result (\ref{etaball}) is implicit in the work of Holt and Vaaler \cite{HV}. Since the proof of this result does not require the full force of the Holt-Vaaler machinery we will provide a self contained proof here. 

\begin{proof}  Suppose $F(\bz)$ is an admissible function for $\eta(B)$. By averaging over $SO(N)$ we find that 
	\[
		\dint_{\Rn} F(x) d{\bf x}=\dint_{\Rn}\dint_{SO(N)}  F({\bf gx}) d\mu({\bf g})d{\bf x}
	\]
where $\mu$ is the normalized Haar measure on $SO(N)$, and that the function 
	\[
		{\bf x}\mapsto \dint_{SO(N)}  F({\bf gx}) d\mu({\bf g})
	\]
is admissible.  In view of this observation we can safely limit our search to extremal functions that are \textit{radial}.  We will see momentarily that the extremal function we find can be factored as $ F(\bz)=U(\bz)U^{*}(\bz)$ where  $ U(\bz)$ is square integrable and radial on $\Rn$ and $ \widehat{U}(\bxi)$ is supported in $\hf B$. This allows us to recast the extremal problem as a minimization problem in a Hilbert space of the form
	\[
		{\bf H }_{\delta}=C(\Rn)\cap\lb { U(\bx)} \in L^{2}(\Rn)\; : \; \widehat{U}(\bxi)=0 \text{ whenever } \xi\not\in\delta B \rb,
	\]
specifically when $\delta=\hf$. The space ${\bf H}_{\delta}$ is a Hilbert space with respect to the $L^{2}(\Rn)$-inner product $\al \cdot  ,   \cdot  \ar$ with the property that for every $\bz\in\Cn$ and $ f\in {\bf H}_{\delta}$  
	\begin{equation}\label{reproIdentity}
		 f(\bz)=\al f,K(\bz,\cdot)\ar
	\end{equation}
where 
	\begin{equation}
		 K(\bo,\bz)=\dint_{\delta B}e^{-2\pi i(\bz-\ol{\bo})\cdot\bxi}d\bxi.
	\end{equation}
  We identify the elements of ${\bf H}_{\delta}$ with their entire extensions to $\Cn$. Let $ H_{\delta}$ be the 1-dimensional case of $\bf H_{\delta}$, that is $H_{\delta}={\bf H}_{\delta}$ when $N=1$. Functions in $H_{\delta}$ which are real-valued and non-negative on the real axis enjoy a factorization akin to that for non-negative trigonometric polynomials given by the Fej\'er-Riesz theorem. The following proposition\footnote{This proposition, due to Ahiezer \cite{Ahiezer, Boas, dB}, is essentially the original Fej\'er-Riesz theorem \cite{RSZ}. } is of central importance in the establishment of (\ref{etaball}), because it allows us to take an awkward $L^1$-minimization problem and reformulate it as a minimization problem in Hilbert space.

	\begin{proposition}\label{fejerThm}
		Suppose $F(z)\in  H_{\delta}$ is real valued and non-negative on the real axis and that $F(z)$ is not identically zero. 
		Then there exists an entire function $U(z)\in  H_{\delta/2}$ such that $U(z)$ is zero-free in $\mU$ and $F(z)=U(z)U^{*}(z)$.
		If $F(z)$ is also even, then $F(z)$ admits the factorization
			\[
				F(z)=z^{2k}Q(z)V(z)V^{*}(z)
			\]
		where $k$ is the multiplicity of the possible zero at $z=0$, $Q(z)$ has only purely imaginary zeros, and $V(z)$ is even.
	\end{proposition}

	\begin{proof}
	Let $\lb \omega_{n} : n=1,2...\rb$ be the zeros of $F(z)$, listed with appropriate multiplicity, in the upper half plane and let 
		\[
			B_{N}(z)=\dprod_{n=1}^{N}\dfrac{1-z/\ol{\omega_{n}}}{1-z/\omega_{n}}.
		\]
	We define a sequence of entire functions $F_{N}(z)$ by $F_{N}(z)=B_{N}(z)F(z)$. Each of the functions $F_{N}(z)$ is in $ H_{\delta}$ by the Paley-Wiener theorem. Since $\|F\|=\|F_{N}\|$ for each $N$, it follows that a subsequence of $F_{N}$ converges weakly to some $G(z)$ in the Hilbert space. By (\ref{reproIdentity}) it follows that $F_{N}(z)\ra G(z)$ pointwise for a subsequence. Since $|B_{N}(z)|\geq 1$ if $z\in\mU$ with equality when $z$ is real, it follows that $G(z)$ is zero free in $\mU$ and that $|G(t)|=|F(t)|$ for real $t$. This shows that $F(z)^{2}=F(z)F^{*}(z)=G(z)G^{*}(z)$. In particular the non-real zeros of $G(z)$ occur with even multiplicity. \newline
\indent Since $F(z)$ is real valued and non-negative on $\R$, the zeros of $G(z)$ occur with even multiplicity and so there is an entire function $U(z)$ for which $G(z)=U(z)^{2}$. Then $F(z)^{2}=\lb U(z)U^{*}(z) \rb^{2}$ and since $F(z)$ is real valued and non-negative on $\R$ it follows that $F(z)=U(z)U^{*}(z)$. \newline
\indent If $F(z)$ is even, write $U(z)=z^{k}p(z)R(z)R^{*}(-z)$ where: $R(z)$ contains the zeros of $U(z)$ which have strictly positive real part, $p(z)$ contains only purely imaginary zeros, and $k$ is the multiplicity of the zero at 0. Let $V(z)=R(z)R(-z)$ and $Q(z)=p(z)p^{*}(z)$. 
\end{proof}

We now introduce a notation for restrictions and extensions for radial functions. If the restriction of $G(\bz)$ to $\Rn$ is radial,  we let $g(z)$ denote its restriction to one of the coordinate axes. Similarly if $g(z)$ is an even entire function, we may extend $g(z)$ to a radial function $G(\bz)$ on $\Cn$ by
		\[
			G(\bz)=\dsum_{\ell=0}^{\infty}\dfrac{g^{(2\ell)}(0) }{(2\ell)!}\big\lbrace z_{1}^{2}+\cdots+z_{N}^{2} \big\rbrace^{\ell}.
		\]
Let $F(\bz)$ be an admissible function for our problem and assume that $F(\bz)$ is radial. Then the corresponding restriction $f(z)$ is an even function in $ H_{1}$ that  is real-valued and non-negative on the real axis. Therefore $f(z)$ admits the representation
		\[
			f(z)=q(z)v(z)v^{*}(z) 
		\]
where $q(z)$ and $v(z)$ are even entire functions and $q(z)$ has only purely imaginary zeros. We choose the functions in such a way that $|v(0)|^{2}=q(0)=1$. Seeing that $q(z)$ and $v(z)$ are even, we extend them to $\Cn$ to obtain the following factorization for $F(\bz)$
		\[
			 F(\bz)=Q(\bz)V(\bz)V^{*}(\bz).
		\]
The integral of $ F(\bx)$ now has the form
		\[
			 \dint_{\Rn}F(\bx)d\bx = \dint_{\Rn}  Q(\bx)|V(\bx)|^{2}d\bx
		\]
But if $ F(\bx)$ is extremal, then $ q(z)$ is zero free. Suppose, by way of contradiction, that $q(z)$ has a zero at say $iy$ for $y>0$. Then
		\[
			 q(z)=\lp1+\dfrac{z^{2}}{y^{2}}\rp\tilde{q}(z)
		\]
for some even entire function $\tilde{q}(z)$ such that $\tilde{q}(0)=1$, and $\tilde{q}(x)\geq 0$ for real $x$. In particular, $\tilde{q}(x)< q(x)$ for all non-zero real numbers $x$. This plainly shows that the admissible function $\tilde{F}(\bz)=\tilde{Q}(\bz)V(\bz)V^{*}(\bz)$ has smaller $L^{1}$-norm than $F(\bz)$. Therefore we may assume
		\[
			 F(\bz)=V(\bz)V^{*}(\bz)
		\]
where $V(\bx)\in {\bf H}_{\hf}$. But by the Cauchy-Schwarz Inequality and (\ref{reproIdentity})
		\[
			1\leq F({\bf 0})= |V({\bf 0})|^{2}\leq K({\bf 0,0}) \| V\|_{2}^{2}=\vol(\hf B)  \| V\|_{2}^{2}.
		\]
where equality occurs if and only if $F({\bf 0})=1$ and $ V(z)$ is a scalar multiple of $ K({\bf 0,z})$. But 
		\[
			  \|V\|_{2}^{2} =\dint_{\Rn}   F(\bx)d\bx.
		\]
Therefore 
		\[
			\eta(B)= \dfrac{2^{N}}{\vol(B)}. 
		\]
		
 \noindent {\bf Now we will show (\ref{etaq}),} but for $Q=[-1,1]^{N}$.\newline 
 
Suppose that $F(\bx)$ is an admissible function for $\eta(Q)$. Then by the Poisson summation formula (see, for instance, \cite{SW})
	\[
		1\leq\dsum_{\bn\in\Z^N}F(\bn)=\dsum_{\bm\in\Z^N}\widehat{F}(\bm)= \widehat{F}(0)= \dint_{\Rn} F(\bx)d\bx.
	\]
We note that both expressions in the Poisson summation formula converge absolutely by a classical result of Poly\'a and Plancherel \cite{pp}. By taking the infimum over all admissible functions $F(\bx)$, we find that $\eta(Q)\geq 1$. But the function
	\[
		F(\bx)=\dprod_{n=1}^{N}\lb  \dfrac{\sin\pi x_{n} }{\pi x_{n}} \rb^{2}.
	\]
is admissible for $\eta(Q)$ and integrating $F(\bx)$ one variable at a time, we find that its integral is equal to 1. This shows $\eta(Q)=2^{N}\vol(Q)^{-1}=1$.\newline

\end{proof}

%\begin{remark}
%	A Hilbert space $H$ which is nontrivial and whose elements are {\it entire functions} is called a {\it de Branges space} if (i) $F(z)\in H$ and $\omega$ is a non-real zero of $F(z)$, then $(z-\ol{\omega})F(z)/(z-\omega)\in H$ and has the same norm as $F(z)$, (ii) $F(z)\in H$ implies $F^{*}(z)\in H$ and has the same norm as $F(z)$, and (iii) for every $\omega\in\C$, then functional $F\mapsto F(\omega)$ is continuous. In view of these properties, Proposition \ref{fejerThm} generalizes almost without change to non-negative functions in a de Branges space since the only properties of the space ${\bf H}_{1}$ we used in the proof are the properties (i)-(iii) above.
%\end{remark}

\section*{Acknowledgments.}%\mcomment{I changed author to second author, but you probably can find a better way of writing it}
Michael Kelly would like to thank Emanuel Carneiro, the faculty, and staff of the Instituto Nacional Matem\'{a}tica Pura e Aplicada (IMPA) in Rio de Janeiro, Brazil, where most of this work was performed. He would also like to thank William Beckner, Hermann K\" onig, Keith Rodgers, Fernando Shao, and Kannan Soundararajan for their remarks and for many stimulating discussions.
%The author offers his most sincere thanks to Gabriele Binachi for helpful suggestions and for showing him how to obtain the cases of equality.
Sincere thanks are also due for an anonymous reviewer for his many useful comments and suggestions which have been taken into account and incorporated into the present version of this paper. Finally Michael Kelly would like to thank Jeffrey Vaaler for his unwavering encouragement and support.

%\nocite{*}
\bibliographystyle{amsalpha}
\bibliography{santalo}	

\providecommand{\bysame}{\leavevmode\hbox to3em{\hrulefill}\thinspace}
\providecommand{\MR}{\relax\ifhmode\unskip\space\fi MR }
\providecommand{\MRhref}[2]{%
  \href{http://www.ams.org/mathscinet-getitem?mr=#1}{#2}
}
\providecommand{\href}[2]{#2}
\begin{thebibliography}{NPRZ10}

\bibitem[Ahi48]{Ahiezer}
N.~I. Ahiezer, \emph{{On the theory of entire functions of finite degree}},
  Doklady Akad. Nauk SSSR (N.S.) \textbf{63} (1948), 475--478. \MR{0027333
  (10,289h)}

\bibitem[Bia13]{Bia13}
Gabriele Bianchi, \emph{{The covariogram and Fourier-Laplace transform in
  $\mathbb{C}^n$}}, 2013, arXiv:1312.7816 [math.MG].

\bibitem[Bla17]{Blaschke1917}
W.~Blaschke, \emph{{{\"U}ber affine Geometrie VII: Neue Extremeingenschaften
  von Ellipse und Ellipsoid}}, Ber. Verh. S{\"a}chs. Akad. Wiss., Math. Phys.
  Kl. \textbf{69} (1917), 412--420.

\bibitem[BM87]{BM}
J.~Bourgain and V.~D. Milman, \emph{{New volume ratio properties for convex
  symmetric bodies in {${\bf R}^n$}}}, Invent. Math. \textbf{88} (1987), no.~2,
  319--340. \MR{880954 (88f:52013)}

\bibitem[Boa54]{Boas}
Jr. Ralph~Philip Boas, \emph{{Entire functions}}, Academic Press Inc., New
  York, 1954. \MR{0068627 (16,914f)}

\bibitem[Cas92]{C}
J.~W.~S. Cassels, \emph{{Obituary: {K}urt {M}ahler}}, Bull. London Math. Soc.
  \textbf{24} (1992), no.~4, 381--397. \MR{1165384 (93f:01016)}

\bibitem[dB68]{dB}
Louis de~Branges, \emph{{Hilbert spaces of entire functions}}, Prentice-Hall
  Inc., Englewood Cliffs, N.J., 1968. \MR{0229011 (37 \#4590)}

\bibitem[HV96]{HV}
Jeffrey~J. Holt and Jeffrey~D. Vaaler, \emph{{The {B}eurling-{S}elberg extremal
  functions for a ball in {E}uclidean space}}, Duke Math. J. \textbf{83}
  (1996), no.~1, 202--248. \MR{1388849 (97f:30038)}

\bibitem[Kim13]{Kim2013}
Jaegil Kim, \emph{{Minimal volume product near Hanner polytopes}}, Journal of
  Functional Analysis (2013), no.~0, --.

\bibitem[Mah39]{M}
Kurt Mahler, \emph{{Ein {M}inimalproblem f{\"u}r {K}onvexe {P}olygone}},
  Mathematica (Zutphen) \textbf{B} (1939), 118--127.

\bibitem[MP90]{MeyPaj1990}
Mathieu Meyer and Alain Pajor, \emph{{On the {B}laschke-{S}antal{\'o}
  inequality}}, Arch. Math. (Basel) \textbf{55} (1990), no.~1, 82--93.

\bibitem[MR89]{MR}
M.~Meyer and S.~Reisner, \emph{{Characterizations of ellipsoids by
  section-centroid location}}, Geom. Dedicata \textbf{31} (1989), no.~3,
  345--355. \MR{1025195 (90m:52006)}

\bibitem[Naz12]{N}
Fedor Nazarov, \emph{{The {H}{\"o}rmander {P}roof of the {B}ourgain-{M}ilman
  {T}heorem}}, {Geometric Aspects of Functional Analysis} (Bo'az Klartag,
  Shahar Mendelson, and Vitali~D. Milman, eds.), {Lecture Notes in
  Mathematics}, vol. 2050, Springer Berlin Heidelberg, 2012, pp.~335--343.

\bibitem[NPRZ10]{NPRZ}
Fedor Nazarov, Fedor Petrov, Dmitry Ryabogin, and Artem Zvavitch, \emph{{A
  remark on the {M}ahler conjecture: local minimality of the unit cube}}, Duke
  Math. J. \textbf{154} (2010), no.~3, 419--430. \MR{2730574 (2012a:52010)}

\bibitem[Pet85]{Petty1985}
C.~M. Petty, \emph{{Affine isoperimetric problems}}, {Discrete geometry and
  convexity ({N}ew {Y}ork, 1982)}, {Ann. New York Acad. Sci.}, vol. 440, New
  York Acad. Sci., New York, 1985, pp.~113--127. \MR{809198 (87a:52014)}

\bibitem[PP37]{pp}
M.~Plancherel and G.~P{\'o}lya, \emph{{Fonctions enti{\`e}res et int{\'e}grales
  de fourier multiples}}, Comment. Math. Helv. \textbf{10} (1937), no.~1,
  110--163. \MR{1509570}

\bibitem[Rei85]{Reisner1}
Shlomo Reisner, \emph{{Random polytopes and the volume-product of symmetric
  convex bodies}}, Math. Scand. \textbf{57} (1985), no.~2, 386--392. \MR{832364
  (87g:52011)}

\bibitem[Rei86]{Reisner2}
\bysame, \emph{{Zonoids with minimal volume-product}}, Math. Z. \textbf{192}
  (1986), no.~3, 339--346. \MR{845207 (87g:52022)}

\bibitem[RSN55]{RSZ}
Frigyes Riesz and B{\'e}la Sz.-Nagy, \emph{{Functional analysis}}, Frederick
  Ungar Publishing Co., New York, 1955, Translated by Leo F. Boron. \MR{0071727
  (17,175i)}

\bibitem[RZ]{RyaZva}
Dmitry Ryabogin and Artem Zvavitch, \emph{{Analytic methods in convex
  geometry}}, in preparation.

\bibitem[{Sai}81]{SR}
J.~{Saint Raymond}, \emph{{Sur le volume des corps convexes sym{\'e}triques}},
  {Initiation {S}eminar on {A}nalysis: {G}. {C}hoquet-{M}. {R}ogalski-{J}.
  {S}aint-{R}aymond, 20th {Y}ear: 1980/1981}, {Publ. Math. Univ. Pierre et
  Marie Curie}, vol.~46, Univ. Paris VI, Paris, 1981, pp.~Exp. No. 11, 25.
  \MR{670798 (84j:46033)}

\bibitem[San49]{Santalo1949}
L.~A. Santal{\'o}, \emph{{Un invariante afin para los cuerpos convexos del
  espacio des n dimensiones}}, Portugaliae Math. \textbf{8} (1949), 155--161.
  \MR{0039293 (12,526f)}

\bibitem[Sch14]{S2}
Rolf Schneider, \emph{{Convex bodies: the {B}runn-{M}inkowski theory}},
  expanded ed., {Encyclopedia of Mathematics and its Applications}, vol. 151,
  Cambridge University Press, Cambridge, 2014. \MR{3155183}

\bibitem[SW71]{SW}
Elias~M. Stein and Guido Weiss, \emph{{Introduction to {F}ourier analysis on
  {E}uclidean spaces}}, Princeton University Press, Princeton, N.J., 1971,
  Princeton Mathematical Series, No. 32. \MR{0304972 (46 \#4102)}

\bibitem[Tao08]{Tao}
Terence Tao, \emph{{Structure and randomness}}, American Mathematical Society,
  Providence, RI, 2008, Pages from year one of a mathematical blog. \MR{2459552
  (2010h:00002)}

\end{thebibliography}

\end{document}